\documentclass[final,1p,times,authoryear]{elsarticle}
\usepackage[T1]{fontenc} % if needed
\usepackage[nolist,printonlyused]{acronym}
\usepackage{amsthm, amssymb, amsmath}
\usepackage{mathtools}
\usepackage{tikz}
\usetikzlibrary{matrix}
\usepackage{enumitem}
\usepackage[english]{babel}
\usepackage{float}
\usepackage{multicol}
\setlist{nosep}

\usepackage{hyperref}
\usepackage{prettyref}
\definecolor{darkgreen}{rgb}{0,0.5,0}
\hypersetup{final, pdfpagemode=FullScreen, colorlinks=true,citecolor=darkgreen}

\usepackage{natbib}

\newcommand{\QQ}{\mathbb{Q}}

\newcommand{\CC}{\mathbb{C}}
\newcommand{\PP}{\mathbb P}

\newcommand{\A}{\mathbb A}

\newcommand{\Qbar}{\overline{Q}}

\DeclareMathOperator{\Frac}{Frac}

\DeclarePairedDelimiter{\ceil}{\lceil}{\rceil}

\newtheorem{theorem}{Theorem}
\newtheorem{lemma}[theorem]{Lemma}
\newtheorem{corollary}[theorem]{Corollary}

\newtheorem{proposition}[theorem]{Proposition}
\newtheorem{definition}[theorem]{Definition}

\newtheorem{example}[theorem]{Example}
\newtheorem{remark}[theorem]{Remark}

\newtheorem{acknowledgment}[]{Acknowledgment}
\begin{document}

\begin{frontmatter}

\title{Rationalizability of square roots}

\author{Marco Besier}
\address{Institut f\"ur Mathematik, Johannes Gutenberg-Universit\"at Mainz, 55099 Mainz, Germany.\\
PRISMA Cluster of Excellence, Institut f\"ur Physik, Johannes Gutenberg-Universit\"at Mainz, 55099 Mainz, Germany.}
\ead{marcobesier@icloud.com}
%\ead[]{}

\author{Dino Festi}
\address{Institut f\"ur Mathematik, Johannes Gutenberg-Universit\"at Mainz, 55099 Mainz, Germany.\\
Dipartimento di Matematica 'Federigo Enriques', Universit\`a degli studi di Milano, via Saldini 50, 20129 Milano, Italy. }
\ead{dinofesti@gmail.com}
%\ead[]{}

\begin{abstract}
Feynman integral computations in theoretical high energy particle physics
frequently involve square roots in the kinematic variables. 
Physicists often want to solve Feynman integrals in terms of multiple polylogarithms.
One way to obtain a solution in terms of these functions is to rationalize all occurring square roots by a suitable variable change.  
In this paper, we give a rigorous definition of rationalizability for square roots of ratios of polynomials.
We show that the problem of deciding whether a single square root is rationalizable can be reformulated in geometrical terms.
Using this approach, 
we give easy criteria to decide rationalizability in most cases of square roots  in one and two variables.
We also give partial results and strategies to prove or disprove rationalizability of sets of square roots.
We apply the results to many examples from actual computations in high energy particle physics. 
\end{abstract}

\begin{keyword}
Feynman integrals, multiple polylogarithms, rationality problems.
\end{keyword}
\end{frontmatter}

\section{Introduction} \label{sec:Introduction}

The measurements carried out at modern particle colliders require accurate theoretical predictions.
To optimize the precision of these predictions, one has to solve Feynman integrals of increasing complexity.
Using dimensional regularization, one writes a given Feynman integral as a Laurent series.
For many Feynman integrals, each term of their Laurent expansion can be written as a linear combination of multiple polylogarithms.
A representation in terms of these functions is favorable because their analytic structure and numerical implementation (cf. \cite{Vollinga:2004sn,Bauer:2000cp}) are well-understood.
Multiple polylogarithms are iterated integrals with integration kernels like
\begin{equation*}
    \omega_j=\frac{dx}{x-z_j},
\end{equation*}
\noindent where the $z_j$ may depend on the kinematic variables but are independent of the integration variable $x$.
The kernels that appear in the computation of Feynman integrals are, however, often more complicated:
they typically involve various square roots.
For example, one encounters kernels like
\begin{equation*}
    \frac{dx}{\sqrt{(x-z_1)(x-z_2)}}.
\end{equation*}
To still find a representation in terms of multiple polylogarithms, one usually tries to proceed as follows:
\begin{enumerate}
    \item[1.] try to find a variable change that turns all square roots into rational functions;
    \item[2.] use partial fractioning to obtain the desired integration kernels.
\end{enumerate}
Changing variables to rationalize a given set of square roots has, therefore, been a crucial step in many particle physics computations, see \cite{Becchetti:2017abb,Broadhurst:1993mw,Fleischer:1998nb,Aglietti:2004tq,Gehrmann:2018yef,Henn:2013woa,Lee:2017oca,Bork:2019aud,Abreu:2019rpt,Primo:2018zby,Chaubey:2019lum,Gehrmann:2015bfy}.
Especially for Feynman integrals in massless theories with dual conformal symmetry, momentum twistor variables turned out to be an excellent variable choice, cf. \cite{Bourjaily:2018aeq}.
Applications can be found also in combinatorics, see for example \cite{Abl17, Abl19}.
An algorithmic approach to the rationalization problem was brought from mathematics to the physics community in \cite{Besier:2018jen} and recently automated with the {\tt RationalizeRoots} software (cf. \cite{Besier:2019kco}), which is available for {\tt Maple} (cf. \cite{Maple}) and {\tt Mathematica} (cf. \cite{Mathematica}).
\par\vspace{\baselineskip}
While these techniques often lead to a suitable variable change, there are still many practical examples where they do not apply.
For these cases, two questions arise:
\begin{itemize}
    \item Are the methods just not powerful enough to find a suitable variable change?
    \item Is it even possible to rationalize the given square roots? If not, can we prove it?
\end{itemize}

This paper is a first attempt to develop simple yet rigorous criteria that physicists can use to answer these questions.
After giving a rigorous definition of rationalizability (Definition~\ref{d:rationalizability}) that is compatible with the notion of ``change of variables'' used in physics,
we show that the problem can be reduced to studying square roots of (squarefree) polynomials instead of square roots of rational expressions. 
This allows us to translate the problem into an arithmetic geometrical language and give some first general partial results.
These are summarized in the following theorem.
\begin{theorem}\label{thm:Main1}
Let $k$ be any field and
let $W=\sqrt{p/q}$ be a square root of a ratio of polynomials, i.e.,  $p,q\in k[X_1,...,X_n]$ and $q$ non-zero.
Then, the following statements hold:
\begin{enumerate}
    \item There exists a squarefree polynomial $f\in k[X_1,...,X_n]$ such that $W$ is rationalizable if and only if $\sqrt{f}$ is.
    \item There exist two projective varieties over $k$ associated to $\sqrt{f}$, denoted by $\overline{V}$ (the associated hypersurface) and $\overline{S}$ (the associated double cover), such that the following statements are equivalent:
    \begin{enumerate}
        \item $\sqrt{f}$ is rationalizable;
        \item $\overline{V}$ is unirational;
        \item $\overline{S}$ is unirational.
    \end{enumerate}
    \item Assume $k$ is algebraically closed and let $d$ denote the degree of $f$.   
    If $d=1,2$, or if $\overline{V}$ has a singular point of multiplicity $d-1$, then $\sqrt{f}$ is rationalizable.
\end{enumerate}
\end{theorem}

By restricting to the physically relevant case $k=\CC$, and by only considering square roots of polynomials in one or two variables, we can give more precise criteria.

\begin{theorem}\label{thm:Main2}
Let $f$ be a squarefree polynomial of degree $d>0$ over $\CC$ and consider its square root $\sqrt{f}$.
Let $\overline{S}$ denote the double cover associated to $\sqrt{f}$.
Then, the following statements hold:
\begin{enumerate}
    \item If $f$ is a polynomial in one variable, then $\sqrt{f}$ is rationalizable if and only if $d\leq 2$.
    \item If $f$ is a polynomial in two variables and $\overline{S}$ has at most rational simple singularities, then $\sqrt{f}$ is rationalizable if and only if $d\leq 4$.
\end{enumerate}
\end{theorem}

Theorems~\ref{thm:Main1} and~\ref{thm:Main2} summarize a number of statements that are proven in Sections~\ref{sec:foundations} and~\ref{sec:criteria}, respectively.
More precisely, we proceed as follows:
in Subsection~\ref{sec:Rationalizability}, we introduce the  definition of (non-)rationalizability.
The associated hypersurface and double cover are defined in Subsections~\ref{sec:ProjSurf} and~\ref{sec:DoubleCover}.
In these subsections, we also prove that the rationalizability of $\sqrt{f}$ is equivalent to the unirationality of the associated hypersurface and double cover, respectively.
We prove the criteria for square roots of polynomials in one and two variables in Subsection~\ref{sec:OneVariable} and~\ref{sec:TwoVariables}.
In Subsection~\ref{sec:MoreVariables}, we deal with square roots of polynomials in more variables and give a  criterion for square roots of homogeneous polynomials.
Finally, in Subsection~\ref{sec:Alphabets}, we introduce the notion of rationalizability for sets of square roots and give a first partial condition that one can use to disprove it.
We end the section with a discussion on how to prove non-rationalizability for several sets of square roots that are directly related to recent Feynman integral computations in theoretical high energy particle physics.

Throughout the paper, we deliberately try to keep the statements simple so that they are easy to apply in practice; most proofs are kept compact, preferring abstract but short arguments over arguments that might be more elementary but longer.

\begin{acknowledgment}
We want to thank Duco van Straten, Stefan Weinzierl, and Robert M. Schabinger for many fruitful conversations.
Furthermore, we would like to thank Anne Fr\"uhbis-Kr\"uger for a detailed explanation of the {\tt Singular} library {\tt classify2.lib}.
The second author was supported by the grant SFB/TRR 45 at the Johannes Gutenberg University in Mainz.
We also thank the anonymous referees for useful comments and suggestions.
\end{acknowledgment}

\section{Foundations}\label{sec:foundations}
\subsection{Notion of rationalizability} 
\label{sec:Rationalizability}

Let $k$ be any field. 
Consider the polynomial ring $R:=k[X_1,...,X_n]$, and denote its field of fractions by $Q:=k(X_1,...,X_n)=\Frac R$.
If we let $f$ and $g$ be two polynomials in $R$ with $g$ non-zero, then $f/g\in Q$.
Fix an algebraic closure $\overline{Q}$ of $Q$, and consider the quantity $\sqrt{f/g}\in \Qbar$.

\begin{definition}\label{d:rationalizability}
We call the square root $\sqrt{f/g}$ {\it rationalizable} 
if there is a homomorphism of $k$-algebras $\phi\colon Q \to Q$ such that $\phi(f/g)$ is a square in $Q$. 
Otherwise, we say that $\sqrt{f/g}$ is {\it not rationalizable}.
\end{definition}

\begin{remark}
Since $\phi\colon Q\to Q$ is a homomorphism of $k$-algebras and since $Q$ is a field, $\phi$ is in particular a homomorphism of fields and preserves the zero element and the unit. 
This implies that $\phi$ is automatically non-constant and, in particular, injective.
\end{remark}

\begin{remark}
Definition~\ref{d:rationalizability} is motivated by the fact that physicists are looking for a {\em change of variables} that turns $f/g$ into a square while preserving the number of variables, i.e., the number of newly introduced variables should be equal to the number of original variables.
Such a change of variables is a non-constant homomorphism of $k$-algebras from $Q$ to $Q$ and vice versa.
\end{remark}

\begin{example}\label{e:defrationalizability}
    The square root $\sqrt{1-X^2}$ is rationalizable.
    If $k$ has characteristic $2$, then 
    $$
    1-X^2=X^2+1=(X+1)^2
    $$ 
    and hence $\sqrt{1-X^2}=X+1$ 
    (the rationalizing map being the identity).
    If $\mathrm{char}\, k\neq 2$,
    consider $\phi\colon k(X) \to k(X)$
    defined by 
    $$
    X \mapsto \frac{2X}{X^2+1},
    $$
    that is, the map that sends the polynomial $g(X)$ to $g\left(\frac{2X}{X^2+1}\right)$. 
    Then, we have 
    $$\phi (1-X^2)= \left( \frac{X^2-1}{X^2+1}\right)^2.$$
    A square root that is \textit{not} rationalizable is, for example, $\sqrt{1-X^3}$. 
    We will explain how to prove its non-rationalizability in Section \ref{sec:criteria} (cf. Corollary~\ref{c:DegreeCriterion}).
\end{example}

\begin{remark}
  Let us stress that our notion of non-rationalizability only implies that there is no \textit{rational} substitution that rationalizes the square root.
  For example, the substitution
  \begin{equation*}\label{eq:FakeMap}
      X\mapsto \sqrt[3]{-X^2+2X}
  \end{equation*}
  turns $\sqrt{1-X^3}$ into a square in $Q$.
  However, it does not define a homomorphism $Q\to Q$ since $\sqrt[3]{-X^2+2X}\notin Q$.
  Physicists dealing with Feynman integrals will mainly be interested in the existence of \textit{rational} substitutions since these do not introduce new square roots in other parts of their computation.
  Nevertheless, the existence of non-rational (called \textit{radical}) parametrizations is studied and applied to other contexts in~\cite{SS11, SS13} and~\cite{SSV17}.
\end{remark}

\begin{remark}
If $\A^n_k$ denotes the affine space with $R$ as its coordinate ring,
then homomorphisms of $k$-algebras $Q\to Q$ are in one-to-one correspondence with dominant rational maps $\A^n_k\to \A^n_k$  \cite[Theorem I.4.4]{Har77}.
This correspondence is crucial since it will allow us to switch between the algebraic and the geometric point of view, see Proposition~\ref{p:criterion}.
\end{remark}

Definition \ref{d:rationalizability} raises an obvious question:
given a square root $\sqrt{f/g}$, can we determine whether or not it is rationalizable?
To answer this question, we will use tools from algebraic geometry.
In particular, we will relate the rationalizability of a square root to the unirationality of a certain variety associated to it.
Before delving into this, let us prove the following elementary lemma.

\begin{lemma}\label{l:eqvrat}
If $p$ and $q$ denote two non-zero elements of $Q$, then $\sqrt{p}$ is rationalizable if and only if $\sqrt{pq^2}$ is rationalizable.
\end{lemma}
\begin{proof}
Assume that $\sqrt{p}$ is rationalizable. 
Then, there exists a non-constant homomorphism of $k$-algebras $\phi\colon Q\to Q$ such that $\phi(p)=r^2$ with $r\in Q$.
But this means that 
\begin{equation*}
    \phi(pq^2)=\phi(p)\phi(q)^2=r^2\phi(q)^2=(r\phi(q))^2\in Q.
\end{equation*}
Hence, $\sqrt{pq^2}$ is rationalizable.
\par\vspace{\baselineskip}
Conversely, assume that $\sqrt{pq^2}$ is rationalizable. 
Then, there is a non-constant homomorphism of $k$-algebras $\phi\colon Q\to Q$ such that $\phi(pq^2)$ is a square in $Q$.
As $\phi$ is a morphism of $k$-algebras, it follows that $\phi(pq^2)=\phi(p)\phi(q^2)=\phi (p)\phi(q)^2$ is a square in $Q$.
As $\phi$ is a non-constant morphism of $k$-algebras that are fields, it is injective, 
and so $\phi (q)$ is non-zero, and hence it follows that $\phi (p)=\phi (pq^2)/\phi (q)^2$ is a square.
Hence, $\sqrt{p}$ is rationalizable.
\end{proof}

\begin{corollary}\label{c:SqFreePol}
Let $p$ and $q$ be two polynomials in $R$ with $q$ non-zero and consider the fraction $p/q\in Q$.
Then there exists a squarefree polynomial $f\in R$ such that $\sqrt{p/q}$ is rationalizable if and only if $\sqrt{f}$ is.
\end{corollary}
\begin{proof}
By Lemma~\ref{l:eqvrat}, 
the square root $\sqrt{p/q}$ is
rationalizable if and only if the square root 
$\sqrt{q^2\cdot p/q}=\sqrt{p q}$ is.
Write $g:=p\cdot q$ and notice that $g\in R$.
As $R$ is a unique factorization domain, we can write $g=fh^2$, with $f,h\in R$ and $f$ squarefree.
Again by Lemma~\ref{l:eqvrat}, $\sqrt{g}$ is rationalizable if and only if $\sqrt{f}$ is,
proving the statement.
\end{proof}

\begin{remark}\label{r:Squarefree}
Let us stress that ignoring squares in the argument is, in many cases, even mandatory for our criteria to be applicable.
In other words, one often \textit{should} ignore square factors in a square root argument:
in Subsection~\ref{sec:TwoVariables}, we will present rationalizability criteria that require the associated variety of the square root to have at most rational simple singularities.
Simple singularities are, in particular, isolated singularities.
The problem is that, if the square root argument contained a square, the associated varieties 
(see Definition~\ref{def:AssHyp} and~\ref{def:AssDoubleCover}) 
would have non-isolated singularities, i.e., a singular locus of positive dimension, and our criteria would not be applicable.
\end{remark}

\subsection{From a square root to a projective hypersurface}
\label{sec:ProjSurf}
Recall that $R$ denotes the polynomial ring $k[X_1,...,X_n]$,  
and $Q=\Frac R$ its field of fractions.
We fixed an algebraic closure $\overline{Q}$ of $Q$.
By Corollary~\ref{c:SqFreePol}, we can reduce our study to square roots of squarefree polynomials.
Throughout this subsection, we use $f\in R$ to denote a non-constant squarefree polynomial of degree $d$ and
consider the square root $\sqrt{f}$ in $\overline{Q}$.

\begin{definition}\label{def:AffAssHyp}
%After squaring $W=\sqrt{f}$, 
%we get the equation $W^2=f$ in $R[W]=k[X_1,...,X_n,W]$.
Let $\A_k^{n+1}$ be the affine space over $k$ with coordinates $X_1,...,X_n,W$.
Let $V$ denote the hypersurface in $\A_k^{n+1}$ defined by the equation $W^2=f$, which is an affine variety of dimension $n$.
We call $V$ {\em the affine hypersurface associated to } $\sqrt{f}$.
\end{definition}
\begin{definition}\label{def:AssHyp}
Let $f$ and $V$ be defined as above, 
and let $\PP_k^{n+1}$ be the projective space over $k$ with coordinates $z,x_1,...,x_n,w$, 
where affine and projective coordinates have the following relations:
$X_i=x_i/z$ for $i=1,...,n$ and $W=w/z$.
We denote by $\overline{V}$ the projective closure of $V\subset \A_k^{n+1}$ in $\PP_k^{n+1}$.
We call $\overline{V}$ {\em the hypersurface associated to } $\sqrt{f}$.
The defining equation of $\overline{V}$ is given by
$$
z^{d-2}w^2-z^{d}f(x_1/z,...,x_n/z)=0,
$$
where $d$ denotes the degree of $f$.
\end{definition}

\begin{example}[Elliptic curve]\label{e:EllipticCubic}
    Consider the square root $\sqrt{X(X-1)(X-\lambda)}$ over $\CC$, for some $\lambda\in \CC\setminus\{0, 1\}$.
    Its associated hypersurface is the projective cubic plane curve $\overline{V}\subset\PP^2_\mathbb{C}$ defined by
   $$
   \overline{V}\colon zw^2-x(x-z)(x-\lambda z)=0. 
   $$
   Note that $\overline{V}$ is an elliptic curve in Legendre form.
\end{example}

The next step is to relate the rationalizability of a square root to the arithmetic of its associated hypersurface.
More precisely, we will show that the rationalizability of a square root is equivalent to the unirationality of its associated hypersurface.
We start by recalling the definitions of rationality and unirationality.

\begin{definition}\label{d:UniRat}
Let $Y$ be a variety defined over $k$, let $\overline{Y}$ be its projective closure, and denote their dimension by $N=\dim Y=\dim \overline{Y}$.
We say that $Y$ is {\em rational} over $k$ if there is a birational map 
$\PP^N\dashrightarrow \overline{Y}$.
We say that $Y$ is {\em unirational} over $k$ if there is a rational dominant map $\PP^N\dashrightarrow \overline{Y}$, i.e., a rational map $\PP^N\dashrightarrow \overline{Y}$ with dense image.
\end{definition}

\begin{remark}\label{r:AlgRat}
One can also rephrase the notions of rationality and unirationality in algebraic terms, cf.~\cite[Theorem I.4.4 and Corollary I.4.5]{Har77}:
a variety $Y$ is rational if its function field $K(Y)$ is isomorphic to a pure transcendental extension field of $k$ of finite type;
it is unirational if its function field can be embedded into a pure transcendental extension field of $k$ of finite type.
The transcendental degree of the extension field equals the dimension of $Y$.
\end{remark}

\begin{remark}\label{r:UniRatEqRat}
Notice that rationality always implies unirationality while the converse statement does, in general, not hold.
There are, however, some special cases in which the two notions are indeed equivalent.
\begin{itemize}
    \item Rationality and unirationality are always equivalent for varieties of dimension one defined over any field (L\"uroth's theorem, cf.~\cite[Example IV.2.5.5]{Har77}).
    \item Over algebraically closed fields of characteristic $0$, the notions of rationality and unirationality are also equivalent for varieties of dimension  two (Castelnuovo's theorem, cf.~\cite[Remark V.6.2.1]{Har77}). 
    This is not true for varieties of higher dimension (see the counterexamples in~\cite{AM72}).
    \item Over non-algebraically closed fields of characteristic $0$ or algebraically closed fields of positive characteristic, things are more complicated: 
    only the equivalence in the one-dimensional case is proven to hold;
    in higher dimensions, the equivalence of unirationality and rationality is either disproved or still disputed, depending on the dimension and the base field.
\end{itemize}
\end{remark}

\begin{proposition}\label{p:criterion}
Let $f$ and $V$ be defined as above.
Then, the square root $\sqrt{f}$ is rationalizable if and only if $V$ is unirational over $k$.
\end{proposition}
\begin{proof}
First assume that $\sqrt{f}$ is rationalizable.
This means that 
there is a non-constant homomorphism of $k$-algebras 
$\phi\colon Q\to Q$ 
such that
$\phi (f)=h^2\in Q$ for some $h\in Q$.
For $i=1,...,n$ let $\phi_i=\phi(X_i)$ denote the image of $X_i$ via $\phi$.
The function field of $V$
is 
$$
K(V)=\Frac\left( \frac{k[X_1,...,X_n,W]}{(W^2-f)}\right).
$$
To show that $V$ is unirational, it suffices to show that $K(V)$ can be embedded into a pure transcendental extension field of $k$ of finite type (cf. Remark~\ref{r:AlgRat}), e.g., embedded into $Q=k(X_1,...,X_n)$.
Consider the homomorphism of $k$-algebras $\Phi\colon K(V) \to Q$ 
defined by sending $X_i$ to $\phi_i$ for $i=1,...,n$, and by sending $W$ to $h$.
As $\phi (f)=h^2$, the map $\Phi$ is well-defined;
as $\phi$ is non-constant, $\Phi$ is non-constant and hence injective, proving the unirationality of $V$.

Conversely, 
assume that $V$ is unirational over $k$. 
As noted in Remark~\ref{r:AlgRat}, this means that there is an injective homomorphism of $k$-algebras 
$$
\Phi\colon K(V)= \Frac\left( \frac{k[X_1,...,X_n,W]}{(W^2-f)}\right) \to Q.
$$
Consider the embedding $\iota\colon Q \to K(V)$ defined by sending the element $X_i$ to its equivalence class in $\Frac\left( \frac{k[X_1,...,X_n,W]}{(W^2-f)}\right)$.
Then $\iota (f)=W^2$ and let $g:=\Phi (W)$ be the image of $W$ via $\Phi$ in $Q$. 
The composition $\Phi\circ\iota\colon Q \to Q$ thus sends $f$ to $\Phi (\iota(f))=\Phi(W^2)=\Phi (W)^2=g^2$,
showing that $\sqrt{f}$ is rationalizable.
\end{proof}

\begin{remark}\label{r:affineprojratequiv}
    The projective closure $\overline{Y}$ of a variety $Y$ is always birationally equivalent to the variety itself.
    Therefore, the unirationality of $\overline{Y}$ is equivalent to the unirationality of $Y$.
    This is reflected in Definition~\ref{d:UniRat},
    where we define the (uni)rationality of any variety (affine or projective) only in terms of its projective closure.
    Thus, one can replace the variety $V$ by its projective closure $\overline{V}$ in the statement of Proposition~\ref{p:criterion}.
\end{remark}

An immediate consequence of Proposition~\ref{p:criterion} is that the rationalizability of a square root only depends on the number of variables that {\em actually appear} in the polynomial, 
and not the total number of variables of the ambient ring.
This is shown in the following corollary.

\begin{corollary}\label{c:Cones}
Let $f\in R$ be defined as before and assume that, after reordering the variables, there exists an $m<n$ such that $f\in R^\prime:=k[X_1,...,X_m]\subset R$.
Then the square root of $f$ viewed as polynomial in $R^\prime$ is rationalizable if and only if the square root of $f$ viewed as polynomial in $R$ is.
\end{corollary}
\begin{proof}
Let $\overline{V}\subset \PP_k^{n+1}$ and $\overline{V^\prime} \subset \PP_k^{m+1}$.
Since the variables $X_{m+1},...,X_{n}$ do not appear in $f$,
the hypersurface $\overline{V}$ is birationally equivalent to $\overline{V^\prime}\times \PP_k^{n-m}$.
So $\overline{V}$ is unirational if and only $\overline{V^\prime}\times \PP_k^{n-m}$ is, 
which in turn is unirational if and only if $\overline{V^\prime}$ is.
The statement hence follows from Proposition~\ref{p:criterion}.
\end{proof}

Using Proposition~\ref{p:criterion}, we can determine the (non-)rationalizability of a square root by studying the unirationality of its associated hypersurface.
Doing this is, however, still a highly non-trivial task.
While the unirationality of varieties is well-studied for one- and two-dimensional varieties over algebraically closed fields of characteristic $0$ , the unirationality of varieties of higher dimension is largely not understood (cf. Remark~\ref{r:GeoRef}).
For this reason, in Section~\ref{sec:criteria}, we will assume $k=\mathbb{C}$ and focus on square roots in one or two variables.
Nevertheless, there are also some partial results holding in any characteristic. 
For example, we have the following corollary.

\begin{corollary}\label{c:SingProj}
Assume $k$ to be algebraically closed.
If $\overline{V}$ has a point of multiplicity $d-1$, then $\sqrt{f}$ is rationalizable.
In particular, if $d\leq 2$, then $\sqrt{f}$ is rationalizable.
\end{corollary}
\begin{proof}
If $d>2$ then $\overline{V}$ is a variety of degree $d$. 
Assume $\overline{V}$ has a point $P$ of multiplicity $d-1$. 
Then the projection from $P$ gives a birational map  $\mathbb{P}_k^N\dashrightarrow\overline{V}$.
Therefore $\overline{V}$ is rational, hence unirational (cf. Remark~\ref{r:UniRatEqRat}) and so, by Proposition~\ref{p:criterion}, $\sqrt{f}$ is rationalizable.
For $d=1,2$, the variety $\overline{V}$ is of degree $2$ and so any of its regular points is  of multiplicity $1$. 
One can then project from any of these points and reason as above to prove the statement.
\end{proof}

\begin{remark}
For a detailed discussion of Corollary~\ref{c:SingProj} and a software package that returns an explicit rational parametrization of a degree-$d$ hypersurface with a point of multiplicity $d-1$, see~\cite{Besier:2018jen} and~\cite{Besier:2019kco}.
\end{remark}

\begin{remark}\label{r:GeoRef}
For the interested reader, the literature on rationality of varieties in dimension~$1$ and~$2$ is vast and starts from the 19th century. 
For a modern account see, among others,
\cite{ACGH85} and~\cite[Chapter IV]{Har77} for curves;
~\cite[Chapters V]{Har77}, and~\cite[Chapters 1---3]{KSC04} for surfaces. 
It is important to remark that the arithmetic of surfaces is much more complicated than the one of curves, 
and the mathematical tools needed to a comprehensive study of it go well beyond the scope of this article.
\end{remark}

\subsection{From a square root to a double cover}
\label{sec:DoubleCover}

In addition to the associated hypersurface defined in the previous subsection, one can also associate another variety to a square root of a polynomial.
We will see that these two varieties are not isomorphic in general but always birationally equivalent to each other.
Hence, in view of our rationalizability study, the two approaches are equivalent.
The approach described in this subsection is particularly convenient for a more geometrical analysis;
the approach described in Subsection~\ref{sec:ProjSurf} is more suitable for a generalization and, a priori, requires less advanced geometrical tools.
Both approaches have advantages and disadvantages in different contexts, and we will use both throughout this paper.

In this subsection,
let $k$ be any field and $f\in R=k[X_1,...,X_n]$ be a non-constant squarefree polynomial of degree $d$.
Define $r:=\ceil{d/2}$.
Consider the square root  $\sqrt{f}$.
Let $\A_k^{n+1}$ be the affine space over $k$ with variables $X_1,...,X_n,W$.
Let $\PP_k=\PP_k(1,...,1,r)$ be the weighted projective space over $k$ with coordinates $s,y_1,...,y_n,u$ of weights $1,1,...,1,r,$ respectively.
The relations between the affine and projective coordinates are $X_i=y_i/s$ for $i=1,...,n$ and $W=u/s$.

\begin{definition}\label{def:AssDoubleCover}
We define {\em the double cover associated to} $\sqrt{f}$ to be the hypersurface in $\PP_k$ given by 
$$
\overline{S}\colon u^2-s^{2r}f(y_1/s,...,y_n/s)=0.
$$
\end{definition}
\begin{remark}
The associated affine hypersurface $V\subset \A_k^{n+1}$ 
(cf. Definition~\ref{def:AffAssHyp}) 
has a natural structure of double cover of $\A_k^n$.
If $S$ denotes $V$ viewed as a double cover, then
$\overline{S}$ is the projective closure of $S$ in $\PP_k$.
\end{remark}

\begin{proposition}\label{prop:Birational}
Let $\overline{V}$ and $\overline{S}$ be the hypersurface and the double cover associated to $\sqrt{f}$.
Then $\overline{V}$ and $\overline{S}$ are birationally equivalent.
\end{proposition}
\begin{proof}
Define the rational map $\Phi\colon \PP_k \dashrightarrow \PP_k^{n+1}$ via 
$$
\Phi\colon (s,y_1,...,y_n,u)\mapsto (s,y_1,...,y_n,u/s^{r-1}).
$$
Then, $\Phi$ is well-defined over a Zariski open subset of $\PP_k=\PP_k (1,1,...,1,r)$, sends $\overline{S}$ to $\overline{V}$, and admits an inverse,
namely the rational map $\Psi\colon\mathbb{P}_k^{n+1}\dashrightarrow\mathbb{P}_k$ defined by
$$
\Psi\colon (z,x_1,...,x_n,w)\mapsto (z,x_1,...,x_n,z^{r-1}w).
$$
Hence, $\Phi$ is a birational map from $\overline{S}$ to $\overline{V}$.
\end{proof}

\begin{example}
Take $k$ to be the field of complex numbers $\CC$. 
In Example~\ref{e:EllipticCubic}, we have seen that the associated hypersurface of the square root $\sqrt{X(X-1)(X-\lambda)}$ is the elliptic curve 
$\overline{V}\subset\PP_k^2$ defined by
$$
\overline{V}\colon zw^2-x(x-z)(x-\lambda z)=0.
$$

The double cover associated to the square root is the curve $\overline{S}$ in the weighted projective space $\PP_k(1,1,2)$ with coordinates $s,y,u$ of weights $1,1,2$, respectively,
defined by the equation
$$
\overline{S}\colon u^2-sy(y-s)(y-\lambda s)=0.
$$
Using the map $\overline{S}\to\PP_k^1$ defined by  $(s:y:u)\mapsto (s:y)$, 
one sees that $\overline{S}$ is a double cover of $\PP_k^1$ ramified at four points: $(0:1:0), (1:0:0), (1:1:0), (1:\lambda:0)$. Hence, $\overline{S}$ is an elliptic curve (by Hurwitz's theorem, cf.~\cite[Corollary IV.2.4]{Har77}).
Finally, Proposition~\ref{prop:Birational} shows that $\overline{S}$ and $\overline{V}$ are indeed birationally equivalent, although a priori they might look very different.
\end{example}
\section{Rationalizability criteria}\label{sec:criteria}
In the subsections that follow, we will always assume the base field to be $k=\mathbb{C}$ and use the shorthand notations $\mathbb{A}^n:=\mathbb{A}^n_k$, $\mathbb{P}^n:=\mathbb{P}^n_k$, and $\mathbb{P}:=\mathbb{P}_k$.

\subsection{Square roots in one variable} 
\label{sec:OneVariable}
Studying the rationalizability of square roots of polynomials in one variable is rather easy:
it all boils down to computing the geometric genus of the curve associated to the square root.
Until the end of this subsection, $f$ will always be a squarefree polynomial in $\CC [X]$ of degree $d>0$.
Let $\A^2$ be the affine plane with coordinates $X$ and $W$. 
Let $C$ denote the affine curve associated to the square root $\sqrt{f}$, and let $\PP^2$ denote the projective plane with coordinates $z,x,w$ and relations $X=x/z, \; W=w/z$.
We write $\overline{C}$ for the projective closure of $C$ in $\PP^2$.

\begin{theorem}\label{thm:CriterionForCurves}
The square root $\sqrt{f}$ is rationalizable if and only if $\overline{C}$ has geometric genus~$0$.
\end{theorem}
\begin{proof}
This immediately follows from Remark~\ref{r:UniRatEqRat} and Proposition~\ref{p:criterion},
keeping in mind that a curve over $\CC$ is rational if and only if it has geometric genus~$0$ (cf.~\cite{lueroth1876,Clebsch:genus}).
For a modern reference, see \cite[Theorems 4.11, 4.62]{Sendra:2008}.
\end{proof}

\begin{remark}
After establishing the existence of a parametrization of the curve,
the next natural question is about the possibility of explicitly providing it.
Whenever a curve $C$ has geometric genus $0$, one can find a rational parametrization (cf.~\cite[p. 133]{Sendra:2008}).
\end{remark}

In fact one can also decide the rationalizability of $\sqrt{f}$ just by looking at the degree of $f$, as shown by the following corollary.
Note that, for this result to hold, it is crucial to assume that $f$ is a {\em squarefree} polynomial.

\begin{corollary}\label{c:DegreeCriterion}
    The square root $\sqrt{f}$ is rationalizable if and only if $d\leq 2$.
\end{corollary}
\begin{proof}
    Following Subsection~\ref{sec:DoubleCover}, 
    let $\overline{S}$ be the double cover associated to $\sqrt{f}$, i.e.,
    the double cover of $\PP^1$ ramified above the zeros of $f$ and, if $d$ is odd, over the point at infinity.
    Then the Riemann--Hurwitz formula (cf.~\cite[Corollary IV.2.4]{Har77})  tells us that 
    $$
     2g(\overline{S})-2=2(2g(\PP^1)-2)+\sum_{P\in \overline{S}} (e_P-1),
     $$
     where $e_P$ is the ramification index of $P\in \overline{S}$.
     Since $f$ is a separable polynomial of degree $d$, we have that
     $$
     \sum_{P\in \overline{S}} (e_P-1) =
     \begin{cases}
     d & \textit{ if }   d \textit{ is even,}\\
     d+1 & \textit{ if }  d \textit{ is odd.}
     \end{cases}
     $$
     As $g(\PP^1)=0$, the formula yields
     $$
     g(\overline{S})=
     \begin{cases}
     (d-2)/2 & \textit{ if }   d \textit{ is even,}\\
     (d-1)/2 & \textit{ if }  d \textit{ is odd.}
     \end{cases}
    $$
    
    From this one clearly sees that $g(\overline{S})=0$ if and only if $d=1,2$.
    Then the statement follows from Theorem~\ref{thm:CriterionForCurves}.
\end{proof}

\begin{remark}\label{r:OneVariable}
    In practice, Corollary~\ref{c:SqFreePol}, and Corollary~\ref{c:DegreeCriterion} allow us to almost immediately determine the rationalizability of a square root:
    assume we want to determine whether the square root $\sqrt{p/q}$ is rationalizable, with $p,q \in \CC [X]$ any two non-zero polynomials and at least one of them non-constant.
    Consider the polynomial $h:=p\cdot q \in \CC[X]$. Then $\sqrt{p/q}$ is rationalizable if and only if $h$ has at most two zeros with odd multiplicity.
\end{remark}

\subsection{Square roots in two variables} 
\label{sec:TwoVariables}

The criterion to decide whether the square root of a polynomial in one variable is rationalizable or not relies on the computation of the geometric genus of the associated curve.
For surfaces, the situation is analogous with the role of the genus being played by the {\em Kodaira dimension}.

\begin{remark}
Recall that the Kodaira dimension of a projective variety $Y$ is an integer
$\kappa=\kappa(Y)\in \{ -1,0,1,...,\dim Y\}$.
Please note that some sources use $-\infty$ instead of $-1$.
For the precise definition, we refer to~\cite[Chapter 6]{Har77}.
\end{remark}

The surfaces arising in our context are mostly not smooth.
The theory of singular surfaces is extremely rich, and a detailed discussion of the topic goes well beyond the scope of this paper; for an overview, see~\cite{AVGZ85}.
In this subsection, we will only deal with surfaces with mild isolated singularities,
that is, surfaces with at most {\em rational simple} singularities
(also called {\em DuVal} or {\em ADE} singularities).
For the definition and their properties, we refer the reader to~\cite[Chapter 15]{AVGZ85}.
(We want to stress out that a ``rational'' simple singularity does \textit{not} need to be defined over $\QQ$;  
``rational'' only means that by resolving it one gets an exceptional divisor birationally equivalent to $\PP^1$.)

Throughout this subsection, $f$ will always denote a non-constant squarefree polynomial in $\CC [X,Y]$ of degree $d$ (cf. Corollary~\ref{c:SqFreePol} and Remark~\ref{r:Squarefree}). 
Following Definition~\ref{def:AssHyp}, we denote by $\overline{V}$ the hypersurface associated to $\sqrt{f(X,Y)}$,
that is,
the projective surface in $\PP^3$ with coordinates $(z,x,y,w)$ defined by
$$
\overline{V}\colon z^{d-2}w^2-z^{d}f(x/z,y/z)=0.
$$
Finally, recall that the notions of being rational and unirational are equivalent for surfaces, cf. Remark~\ref{r:UniRatEqRat}.

\begin{lemma}\label{l:EqvUniKodDeg}
Let $f, d$, and $\overline{V}$ be defined as above and assume $\overline{V}$ is smooth or has at most rational simple singularities.
Then the following statements are equivalent.
\begin{enumerate}[label=(\roman*)]
    \item $\overline{V}$ has Kodaira dimension $-1$.
    \item the degree of $f$ is at most three, i.e., $d\leq 3$;
    \item $\overline{V}$ is unirational;
    \item $\sqrt{f}$ is rationalizable.
\end{enumerate}
\end{lemma}
\begin{proof}
From Proposition~\ref{p:criterion}, we know that $iii)\iff iv)$.
We are left to show that $i), ii),$ and $iii)$ are equivalent.
Therefore, we will show that $i)\implies ii) \implies iii) \implies i)$.

{\em i) $\implies$ ii)} 
Assume $\overline{V}$ has at most rational simple singularities and $\kappa (\overline{V})=-1$.
Also, recall that $\overline{V}$ is a hypersurface in $\PP^3$ of degree $d$.
Since $\overline{V}$ has at most rational simple singularities, the canonical class is left unchanged after passing to the smooth model, 
and hence the Kodaira dimension of $\overline{V}$ is determined by its degree.
Hypersurfaces in $\PP^3$ of Kodaira dimension $-1$ have degree $d=1,2,3$ (in fact $\kappa(\overline{V})=0$ for $d=4$ and $\kappa(\overline{V})=2$ for $d\geq 5$, cf.~\cite[Section VI.1]{BHPV04}).

{\em ii) $\implies$ iii)}
 If $d=1,2$, then the statement follows from Corollary~\ref{c:SingProj} and Proposition~\ref{p:criterion}.
Assume $d=3$, then $\overline{V}$ is a cubic. 
If it is smooth, then it is rational (classic result by Clebsch, see~\cite{Cle66} for the original paper, or~\cite[Corollay V.4.7]{Har77} for a more modern statement and proof).
Rationality implies unirationality.
If $\overline{V}$ is singular, by assumption the singularities must be rational simple and hence of multiplicity $2=d-1$. 
Then the unirationality of $\overline{V}$ follows again from Corollary~\ref{c:SingProj} and Proposition~\ref{p:criterion}.

{\em iii) $\implies$ i)}
Assume $\overline{V}$ is unirational. As already noted, this is equivalent to saying that $\overline{V}$ is rational.
Then by the Enriques--Kodaira classification of surfaces it follows that $\kappa (\overline{V})=-1$ (see also~\cite[Theorem V.6.1]{Har77}). 
\end{proof}
\begin{remark}\label{r:SingSurf}
Unfortunately, Lemma~\ref{l:EqvUniKodDeg} is not very useful in practice: when $d\geq 4$,
the surface $\overline{V}$ has a non-simple singular point at $(0:0:0:1)$, so the result does not apply.
If $d=1,2$ we already know (unconditionally) that $\sqrt{f}$ is rationalizable  (cf. Proposition~\ref{p:criterion}).
The only interesting case is when $d=3$, as shown by the following corollary.
\end{remark}
\begin{corollary}\label{c:CubicSurface}
Assume $d=3$.
Then $\sqrt{f}$ is rationalizable if and only if $\overline{V}$ has no singular points of multiplicity $3$.
\end{corollary}
\begin{proof}
If $\overline{V}$ has a singular point of multiplicity $3$,
then it is the projective cone over a projective plane cubic curve (with the singular point being the vertex of the cone).
Hence $\overline{V}$ is ruled and not (uni)rational and so, by Proposition~\ref{p:criterion}, it follows that $\sqrt{f}$ is not rationalizable.

Conversely, assume that $\overline{V}$ has no singular points of multiplicity $3$.
As $\overline{V}$ is a cubic, this means that it is either smooth or has singular points of multiplicity $2$.
In the smooth case, the statement follows from Lemma~\ref{l:EqvUniKodDeg};
in the singular case, from Corollary~\ref{c:SingProj}.
\end{proof}
\begin{remark}\label{r:NonSimpleSing}
The assumption in Lemma~\ref{l:EqvUniKodDeg} for $\overline{V}$ to have at most rational simple singularities is strictly necessary, as shown by Corollary~\ref{c:CubicSurface}.
We have seen that if $d=3$ and $\overline{V}$ has a point of multiplicity $3$,
then $\overline{V}$ is a cone over a cubic curve, which is ruled but not (uni)rational, providing a counterexample to $ii)\implies iv)$.
In~\cite{IN04}, one can find a ruled quartic surface having Kodaira dimension $-1$ while not being (uni)rational, hence a counterexample to $i)\implies ii)$.
The implications $iii)\iff iv)$, $iii)\implies i)$, and $ii)\implies i)$ hold unconditionally.
\end{remark}

The guaranteed existence of non-simple singular points on $\overline{V}$ (cf. Remark~\ref{r:NonSimpleSing}) prevents us from getting much information about the rationalizability of $\sqrt{f}$.
However, we can also use the approach of Subsection~\ref{sec:DoubleCover}, which turns out to be much more suitable for the case of square roots in two variables, as shown in the theorem below.
Let $\overline{S}$ denote the double cover associated to $\sqrt{f}$ (cf. Definition~\ref{def:AssDoubleCover}).

\begin{theorem}\label{thm:DegreeCriterion}
Assume that $\overline{S}$ has at most rational simple singularities. 
Then $\sqrt{f}$ is rationalizable if and only if $d\leq 4$.
\end{theorem}
\begin{proof}
From Remark~\ref{r:UniRatEqRat} and Proposition~\ref{p:criterion}, $\sqrt{f}$ is rationalizable if and only if $\overline{S}$ (or, equivalently, $\overline{V}$) is unirational.

So now assume $\overline{S}$ has at most rational simple singularities. 
Let $T$ be the desingularization of $\overline{S}$.
By a result in~\cite{Hir64}, such a $T$ exists and is birationally equivalent to $\overline{S}$.
This means that $\overline{S}$ is (uni)rational if and only $T$ is.
As $\overline{S}$ has only rational simple singularities, the canonical divisor of $T$ equals the canonical divisor of $\overline{S}$. 
Hence, we can use the degree of $f$ to compute the canonical divisor of $\overline{S}$ and hence the Kodaira dimension of $T$. 

If $d=1,2$ then $\sqrt{f}$ is rationalizable by Proposition~\ref{c:SingProj}.

If $d=3,4$, then $T$ (and hence $\overline{S}$)  is birationally equivalent to a del Pezzo surface of degree $2$ (cf.~\cite[Theorem III.3.5]{Kol96}, where the degree of a del Pezzo surface is defined to be the self-intersection of the canonical divisor of the surface; notice that it does not need to coincide with the degree of the defining polynomial).
Del Pezzo surfaces are rational (cf.~\cite[Theorem IV.24.4]{Man86}).

We are left to show that if $d>4$, then $\sqrt{f}$ is not rationalizable. 
In order to see this, we prove that $\overline{S}$ is not (uni)rational.
Since $d>4$, we have that $T$ and, hence, $\overline{S}$ have Kodaira dimension greater than or equal to $0$ 
(in fact, their Kodaira dimension is $0$ if $d=5,6$ and $2$ if $d\geq 7$, cf.~\cite[Section V.22]{BHPV04}).
As $\overline{S}$ and $\overline{V}$ are birationally equivalent, they have the same Kodaira dimension.
Then Lemma~\ref{l:EqvUniKodDeg} implies that $\sqrt{f}$ is not rationalizable, proving the statement.
\end{proof}

\begin{remark}
At a first glance, Theorem~\ref{thm:DegreeCriterion} and  Lemma~\ref{l:EqvUniKodDeg}  contradict each other, as the square root of a polynomial $f$ of degree $d=4$ should be non-rationalizable, according to Lemma~\ref{l:EqvUniKodDeg}, but also rationalizable, 
according to Theorem~\ref{thm:DegreeCriterion}.
This contradiction does, however, not really exist:
as already noted in Remark~\ref{r:NonSimpleSing}, the hypersurface $\overline{V}$ associated to $f$ always has a non-simple singular point and, therefore, one cannot apply the implication $i)\implies ii)$ in Lemma~\ref{l:EqvUniKodDeg} (cf. Remark~\ref{r:NonSimpleSing}) needed to conclude that $\sqrt{f}$ is not rationalizable.
\end{remark}

\begin{remark}
In order to apply Theorem~\ref{thm:DegreeCriterion},
one needs to study the singularities of~$\overline{S}$.
To simplify this task, we wrote a {\tt Magma}~(cf.~\cite{Magma}) function.
For the source code of the function and a detailed explanation of how to apply it, see \cite{Besier:GitHub}.
Alternatively, one can use the {\tt classify2.lib} library of {\tt Singular} (cf.~\cite{singular}).
Both {\tt Magma} and {\tt Singular}  come with a free online calculator that one can use to perform the singularity classification.
\end{remark}

\begin{example}
In~\cite{Festi:092018}, the rationalizability of the square root
\begin{equation}\label{eq:Bhabha}
\sqrt{\frac{(X+Y)(1+XY)}{X + Y - 4XY + X^2Y + XY^2}}
\end{equation}
coming from the Bhabha scattering (cf. \cite{Henn:2013woa}) is studied.
Using Corollary~\ref{c:SqFreePol}, one immediately sees that this is equivalent to study the unirationality of the double cover $\overline{S}$ associated to the square root
$$
\sqrt{(X+Y)(1+XY)(X + Y - 4XY + X^2Y + XY^2)}.
$$
The surface $\overline{S}$ has only simple singularities as one can check either by hand or using our code, cf.~\cite{Besier:GitHub}.
Therefore, from Theorem~\ref{thm:DegreeCriterion}, it follows that the square root~\eqref{eq:Bhabha} is not rationalizable. 
\end{example}

\subsection{Square roots in three or more variables}
\label{sec:MoreVariables}

In the previous subsections, we have seen that, if the argument of the square root is a polynomial in one or two variables, then we can often determine whether the square root is rationalizable or not by investigating the unirationality of the associated varieties.

Unfortunately, the situation becomes dramatically more complicated as the number of variables grows.
Already in the case of three variables, the previous approaches will not work anymore as we lack easy criteria to assess the unirationality of threefolds.
The same is true for varieties of even higher dimension.

In Subsection~\ref{sec:ProjSurf}, we have already seen a partial result to deduce rationalizability in any number of variables (cf. Corollary~\ref{c:SingProj}). 
In this subsection, we give a result that can help practitioners in studying the rationalizability of a square root of a {\em homogeneous} polynomial.
More precisely, we will show how to reduce the study of the square root of a  homogeneous polynomial in $n$ variables to the study of a square root of a (non-homogeneous) polynomial in $n-1$ variables.
This process is particularly interesting when $n=3$,
as we can then use all the results of Subsection~\ref{sec:TwoVariables}.

In what follows we will always assume that $f$ is a non-constant squarefree polynomial of degree $d$ in $R=k[X_1,...,X_n]$;
recall that we fixed $k=\CC$.
We use $\overline{V}$ to denote the hypersurface associated to $\sqrt{f}$ (see Definition~\ref{def:AssHyp}).

\begin{proposition}\label{p:Ruled}
Let $f$ and $d$ be defined as above and 
let $F$ be the homogeneization of $f$ in $k[x_0,x_1,...,x_n]$ with $X_i=x_i/x_0$ for $i=1,...,n$,
that is, $F=x_0^df(x_1/x_0,...,x_n/x_0)$.

The following holds: 
\begin{enumerate}
    \item if $d$ is even, then $\sqrt{f}$ is rationalizable if and only if  $\sqrt{F}$ is;
    \item if $d$ is odd, then $\sqrt{f}$ is rationalizable if and only if  $\sqrt{x_0F}$ is.
\end{enumerate}
\end{proposition}
\begin{proof}
In what follows, let $Q'$ be the field $k(x_0,x_1,...,x_n)$ and recall $Q:=k(X_1,...,X_n)$.
\begin{enumerate}
    \item By assumption $d$ is even; write $d=2r$.
    Assume $\sqrt{F}$ is rationalizable. Then there is a non-constant homomorphism of $k$-algebras $\Phi\colon Q^\prime\to Q^\prime$ such that
$$
\Phi (F) = F (\Phi_0,...,\Phi_n)=H^2,
$$
where $\Phi_i:=\Phi(x_i)$ and $H\in Q^\prime$.
As $F=x_0^df(x_1/x_0,...,x_n/x_0)$ we have
\begin{align*}
    H^2=\Phi(F)&=\Phi(x_0^df(x_1/x_0,...,x_n/x_0) )\\
           &=\Phi (x_0)^d \Phi (f(x_1/x_0,...,x_n/x_0) )\\
           &=\Phi_0^d f(\Phi_1/\Phi_0,...,\Phi_n/\Phi_0)\; ,\\
\end{align*}
from which it follows that
$$
f(\Phi_1/\Phi_0,...,\Phi_n/\Phi_0)=\frac{H^2}{\Phi_0^d}= \left(\frac{H}{\Phi_0^r}\right)^2
$$
is a square in $Q^\prime$.
Notice that,  as $\Phi$ is a homomorphism of fields, it is injective and hence $\Phi_0$ is non-zero.
Also, as $k$ is algebraically closed, it is infinite and hence there exists an element $c\in k$ such that
the rational expressions $H(c,x_1,...x_n)$ and 
 $\Phi_i(c,x_1,...,x_n)$, for $i=0,1,...,n$,  are well-defined and non-zero.
Then the following homomorphism of $k$-algebras is well defined.
\begin{align*}
    \phi\colon Q&\to Q\\
                X_i &\mapsto \frac{\Phi_i(c,X_1,...,X_n)}{\Phi_0(c,X_1,...,X_n)} \text{ for } i=1,...,n
\end{align*}
It is easy to see that $\phi$ sends $f$ to a square, indeed
\begin{align*}
    \phi(f)&=f\left(\frac{\Phi_1(c,X_1,...,X_n)}{\Phi_0(c,X_1,...,X_n)},...,\frac{\Phi_n(c,X_1,...,X_n)}{\Phi_0(c,X_1,...,X_n)}\right)\\
    &=(f(\Phi_1/\Phi_0,...,\Phi_n/\Phi_0))(c,X_1,...,X_n)\\
    &=\frac{H^2}{\Phi_0^d}(c,X_1,...,X_n)\\
    &=\left(\frac{H(c,X_1,...,X_n)}{\Phi_0^r(c,X_1,...,X_n)}\right)^2.
\end{align*}
Hence, $\sqrt{f}$ is rationalizable.

Conversely,
assume that $\sqrt{f}$ is rationalizable.
Then, there is a non-constant homomorphism $\phi\colon Q\to Q$ of $k$-algebras such that 
$$
\phi(f)=f(\phi_1,...,\phi_n)=h^2
$$ 
for some $h\in Q$, where $\phi_i=\phi (X_i)\in Q$ for $i=1,...,n$.
All the $\phi_i$'s can be expressed as a ratio of two polynomials; taking the least common multiple $\varphi_0$ of the denominators, we can write
$$
\phi_i=\frac{\varphi_i}{\varphi_0}
$$
for $i=1,...,n$ (notice that $\varphi_0$ is fixed).
Then, as above, we can find $n+1$ polynomials $\Phi_0,\Phi_1,...,\Phi_n\in k[x_0,x_1,...,x_n]$ (homogeneous and of the same degree) such that the following equalities hold in $Q^\prime$:
$$
\frac{\varphi_i (x_1/x_0,...,x_n/x_0)}{\varphi_0(x_1/x_0,...,x_n/x_0)}=\frac{\Phi_i}{\Phi_0},
$$
for every $i=1,...,n$.
Define the homomorphism $\Phi$ as follows.
\begin{align*}
    \Phi\colon Q^\prime&\to Q^\prime\\
                x_i &\mapsto \Phi_i \text{ for } i=0,1,...,n
\end{align*}
Then $\Phi$ sends $F$ to a square, concluding the proof:
\begin{align*}
    \Phi (F) =& \Phi (x_0^d f(x_1/x_0,...,x_n/x_0))\\
             =& \Phi (x_0)^d \Phi (f(x_1/x_0,...,x_n/x_0))\\
             =& \Phi_0^d f(\Phi (x_1/x_0),...,\Phi (x_n/x_0)))\\
             =& \Phi_0^d f(\Phi_1/\Phi_0,...,\Phi_n/\Phi_0)\\
             =& \Phi_0^d f\left( \frac{\varphi_1 (x_1/x_0,...,x_n/x_0)}{\varphi_0(x_1/x_0,...,x_n/x_0)},..., \frac{\varphi_n (x_1/x_0,...,x_n/x_0)}{\varphi_0(x_1/x_0,...,x_n/x_0)}\right)\\
             =& \Phi_0^d \, \left(f\left( \frac{\varphi_1}{\varphi_0},...,\frac{\varphi_n}{\varphi_0}\right)\right) (x_1/x_0,...,x_n/x_0)\\
             =&\Phi_0^d\,(f(\phi_1,...,\phi_n))(x_1/x_0,...,x_n/x_0)\\
             =&\Phi_0^d h(x_1/x_0,...,x_n/x_0)^2\\
             =&(\Phi_0^r h(x_1/x_0,...,x_n/x_0))^2.
\end{align*}

\item Assume that $d$ is odd and write $d=2r-1$.
Then $x_0F$ has degree $2r$, and the proof goes as above.

\end{enumerate}
\end{proof}

\begin{remark}
Proposition~\ref{p:Ruled} is particularly useful in the case of square roots in three variables.
Indeed, if $n=3$ and $f$ happens to be homogeneous, 
then one can regard $f$ as the homogenization of a polynomial $g$ in two variables.
Then $\sqrt{f}$ is rationalizable if and only if the hypersurface (or, equivalently, the double cover) associated to $\sqrt{g}$ 
(or to $\sqrt{x_0g}$, if $d$ is odd and $x_0$ is the homogenizing variable) is unirational.
Subsequently, one can apply the methods of Subsection~\ref{sec:TwoVariables}.
\end{remark}

\begin{example}
With this example, we show that Proposition~\ref{p:Ruled} can be helpful even with square roots in two variables.
Consider the square root $\sqrt{F}$ with $F=X_1^4+X_2^4$ and let $\overline{S}$ be its associated double cover in $\PP(1,1,1,2)$ with coordinates $s,y_1,y_2,u$, that is,
$$
\overline{S}\colon u^2=y_1^4+y_2^4.
$$
One can see that $\overline{S}$ has a non-simple (elliptic) singularity in $(1:0:0:0)$ and therefore we cannot use Theorem~\ref{thm:DegreeCriterion} to conclude that $\sqrt{f}$ is rationalizable.

Nevertheless, $F$ can be seen as the homogenization of the polynomial $G=X^4+1$ and, by Proposition~\ref{p:Ruled}, $\sqrt{F}$ is rationalizable if and only if $\sqrt{G}$ is (as the degree of $G$ is $4$, even).
As $G$ is a polynomial in one variable, we can then 
apply Theorem~\ref{thm:CriterionForCurves} to conlude that $\sqrt{G}$, and hence also $\sqrt{F}$, is not rationalizable.  
\end{example}

\begin{example}
Fermat quartics give us also another interesting example.
Consider the square root $\sqrt{F}$ with $F=X_1^4+X_2^4+X_3^4$.
Notice that $F$ has degree $4>2$ so we cannot conclude right away from Corollary~\ref{c:SingProj} that $\sqrt{F}$ is rationalizable.
The associated hypersurface is $\overline{V}:=\{ z^2w^2-x_1^4-x_2^4-x_3^4=0\}$
and has two (non-simple) singular points, $(1:0:0:0:0)$ and $(0:0:0:0:1)$, both of multiplicity $2$. In particular, no triple points.
So again, we cannot apply Corollary~\ref{c:SingProj} to conclude that it is rationalizable.

Nevertheless, we notice that $F$ is homogeneous and it can be seen as the homogenization of $f=X^4+Y^4+1$ in $k[X_1,X_2,X_3]$ with $X=X_1/X_3$ and $Y=X_2/X_3$.
Let $\overline{S}$ denote the double cover associated to $\sqrt{f}$.
It is easy to see that $\overline{S}$ is smooth and so, in particular, has no non-simple singularities.
From Theorem~\ref{thm:DegreeCriterion} it follows that $\sqrt{f}$ is rationalizable. 
Hence, by Proposition~\ref{p:Ruled}, so is $\sqrt{F}$.
\end{example}

\subsection{Proving non-rationalizability of a set of square roots} 
\label{sec:Alphabets}

In the context of Feynman integral computations, it is often not enough to study the rationalizability of a single square root. 
Instead, practitioners are usually interested in whether or not several different square roots can be rationalized simultaneously.
However, we will see that the non-rationalizability of a set of square roots (also called \textit{alphabet}) 
can often be deduced from the non-rationalizability of a single square root so that many of our previous methods can also be applied in this more general context. 

As before, we fix $k=\mathbb{C}$, 
and write $R=k[X_1,...,X_n]$ for the ring of polynomials and $Q=\mathrm{Frac}\;R$ for its field of fractions.

\begin{definition}
Let $f_1,...,f_r$ be polynomials in $R$.
An alphabet $\{ \sqrt{f_1},...,\sqrt{f_r}\}$ is called {\em rationalizable}
if there is a homomorphism of $k$-algebras $\phi\colon Q\to Q$ such that $\; \phi(f_i)=h_i^2$ for some $h_i\in Q$, where $i=1,...,r$.
\end{definition}

\begin{remark}
From the definition it immediately follows that if an alphabet $\mathcal{A}^\prime$ is non-ra\-tio\-na\-li\-za\-ble, 
then every alphabet $\mathcal{A}\supseteq \mathcal{A}^\prime$ containing $\mathcal{A}^\prime$ is also non-rationalizable.

This remark is particularly useful when $\mathcal{A}$ has a subset $\mathcal{A}^\prime$ containing only square roots of polynomials in fewer variables, that is, after possibly reordering the variables and the polynomials,
$$
\mathcal{A}^\prime=\Big\{\sqrt{f_1},...,\sqrt{f_s}\Big\}  
$$
with $s<r$ and $f_1,...,f_s\in k[X_1,...,X_m]\subset R$, $m<n$.
Then one can try to disprove the rationalizability of $\mathcal{A}$ by disproving the rationalizability of $\mathcal{A}^\prime$.
The rationalizability of $\mathcal{A}^\prime$ as alphabet of square roots of polynomials in $R$ is equivalent to its rationalizability as square roots of polynomials in $k[X_1,...,X_m]$ (Corollary~\ref{c:Cones}).
The latter task is easier because of the fewer variables involved and, in particular, if $m=1,2$, one can then apply the criteria given in Subsections~\ref{sec:OneVariable} and~\ref{sec:TwoVariables}.
\end{remark}

\begin{proposition}\label{p:MoreRoots}
If the alphabet $\{ \sqrt{f_1},...,\sqrt{f_n}\}$ is rationalizable then, for every non-empty subset $J\subseteq \{1,...,n\}$, the square root 
\begin{equation}\label{eq:productRoot}
\sqrt{\prod_{j\in J} f_j}
\end{equation}
is rationalizable.
\end{proposition}
\begin{proof}
By definition, if $\{ \sqrt{f_1},...,\sqrt{f_n}\}$ is rationalizable, then there exists a non-constant $k$-algebra homomorphism $\phi\colon Q\to Q$ such that, for $i=1,...,n$, the map $\phi$ sends $f_i$ to $h_i^2$, for some $h_i\in Q$.
In particular, $\phi(f_j)=h_j^2$ for every $j\in J$.
Hence, $\phi (\prod_{j \in J}f_j)=(\prod_{j\in J}h_j)^2$,
proving the statement.
\end{proof}
\begin{remark}\label{r:AlphabetStrategy}
The above straightforward proposition allows us to prove that a given alphabet is not rationalizable by showing that at least one square root of the form \eqref{eq:productRoot} is not rationalizable. 

At the moment, we do not know whether the converse statement of Proposition~\ref{p:MoreRoots} holds or not, 
as even if the product $\sqrt{\prod_{j\in J} f_j}$ is rationalizable for every $J\subseteq \{1,...,n\}$,
the rationalizing morphisms $\phi_J$ do not need to be a priori all equal.
Nevertheless, one can often prove the rationalizability of a given alphabet by providing an explicit variable change that rationalizes all of its square roots simultaneously.
To find such a variable change, one can try to apply the elementary strategy mentioned in Remark~\ref{r:StartingPoint}, which works for many practical examples.
\end{remark}

Let us apply Proposition~\ref{p:MoreRoots} to the alphabets of some recent Feynman integral computations.
\begin{example}
By Corollary~\ref{c:DegreeCriterion}, square roots of a squarefree polynomial in one variable of degree $d>2$ are not rationalizable.
Such square roots occurred in many Feynman integral computations of the last decade, see  \cite{Laporta:2004rb,MullerStach:2011ru,Adams:2013kgc,Bloch:2013tra,Adams:2014vja,Adams:2015gva,Adams:2015ydq,Sogaard:2014jla,Bloch:2016izu,Remiddi:2016gno,Adams:2016xah,Bonciani:2016qxi,vonManteuffel:2017hms,Adams:2017ejb,Bogner:2017vim,Ablinger:2017bjx,Remiddi:2017har,Bourjaily:2017bsb,Hidding:2017jkk,Broedel:2017kkb,Broedel:2017siw,Broedel:2018iwv,Adams:2018yfj,Adams:2018bsn,Adams:2018kez}. 
As an example, let us consider the following alphabet, which appears in perturbative corrections for Higgs production in \cite{Dulat:talk2018, Anastasiou:2015yha}:
\begin{equation*}
    {\mathcal A}=
 \Big\{
 \sqrt{X},
 \sqrt{1+4X},
 \sqrt{X (X-4)}\Big\}.
\end{equation*}

These three square roots cannot be rationalized by a single rational variable change.
To see this, define $f_1:=X, f_2:=1+4X, f_3:=X(X-4)$,
take $J:=\{2,3\}$, and consider the square root
\begin{equation}\label{eq:higgsRoot}
\sqrt{\prod_{j\in J}f_j}=\sqrt{(1+4X)X(X-4)}.
\end{equation}
Note that the product $(1+4X)X(X-4)$ is a squarefree polynomial of degree $3>2$.
Therefore, by Corollary~\ref{c:DegreeCriterion}, the square root \eqref{eq:higgsRoot} is not rationalizable.
Thus, by Proposition~\ref{p:MoreRoots}, we conclude that $\mathcal{A}$ is not  rationalizable.
\end{example}

\begin{example}\label{e:dijetproduction}
Let us now consider the following set of square roots that is relevant for perturbative corrections to di-photon and di-jet hadro-production in \cite{Becchetti:2017abb}:

\begin{align*}
    \mathcal{A}=&\Big\{\; \sqrt{X+1},\; \sqrt{X-1}, \; \sqrt{Y+1},\\  & \sqrt{X+Y+1},\; \sqrt{16X+(4+Y)^2}\; \Big\}.
\end{align*}
Write $f_1,..., f_5$ for the polynomial arguments of the square roots in $\mathcal{A}$.
To show that $\mathcal{A}$ is not  rationalizable, consider the whole set of indices $J=\{1, 2, 3, 4, 5\}$
and define
$$
f:=\prod_{j\in J} f_j=(X+1)(X-1)(Y+1)(X+Y+1)(16X+(4+Y)^2).
$$
Let $\overline{S}$ be the associated double cover of $\sqrt{f}$ (cf. Definition~\ref{def:AssDoubleCover}). 
It is easy to check---for example by using our {\tt Magma} function---that $\overline{S}$ has only rational simple singularities.
Since $f$ has degree $6$,
Theorem~\ref{thm:DegreeCriterion} tells us that $\sqrt{f}$ is not rationalizable.
Hence, 
using Proposition~\ref{p:MoreRoots}, 
we can conclude that $\mathcal{A}$ cannot be  rationalizable.
\end{example}

\begin{example}
It is important (and fair) to stress that the results presented in this paper are not always enough to get an answer.
Consider the alphabet
\begin{align*}
    \mathcal{A}=\Big\{ &\sqrt{X_1(X_1-4X_3)}, \sqrt{-X_1X_2(4X_3(X_3+X_2)-X_1X_2)},\\
    &\sqrt{X_1(X_2^2(X_1-4X_3)+X_3X_1(X_3-2X_2))}\,  \Big\}
\end{align*}
 relevant for two-loop EW-QCD corrections to Drell--Yan scattering (cf.~\cite{Heller:2019gkq,Besier:2019hqd,Bonciani:2016ypc}).
Denote by $F_1, F_2, F_3$ the polynomial arguments of the square roots in $\mathcal{A}$.
Proving non-rationalizability of~$\mathcal{A}$ requires more than just the techniques presented in this paper.
Notice that $F_1, F_2, F_3$ are all homogeneous.
Therefore, we can view them as the homogenizations of three polynomials with respect to one of the three variables, for example $X_3$.
Studying the rationalizability of $\mathcal{A}$ is, hence, equivalent to studying the rationalizability of  
$$
\Big\{\; \sqrt{f_1}, \sqrt{f_2}, \sqrt{f_3}\; \Big\},
$$
where $f_i=f_i(X,Y)$ is the dehomogenization of $F_i$ with respect to $X_3$, that is,
$$
f_i(X,Y):=F_i(X,Y,1).
$$

As $f_1, f_2$ and $f_3$ have degree $2, 4$ and $4$ respectively, and since their associated double covers have at most rational simple singularities, one has that $\sqrt{f_1},\sqrt{f_2},\sqrt{f_3}$ are rationalizable when considered individually (cf. Theorem~\ref{thm:DegreeCriterion}).
The products $f_1f_2$ and $f_1f_3$, after removing the square factors, also have degree $4$ and associated double covers with at most rational simple singularities.
Hence, their square roots are rationalizable.
The product $f_2f_3$ has,
after removing square factors, degree $6$ but its associated double cover has (at least) two non-simple singularities, so we cannot conclude that $\sqrt{f_2f_3}$ is not rationalizable.
(After further investigation it turns out that it is in fact rationalizable.)
We are left with the square root of the product $f_1f_2f_3$.
After removing the squares, 
the product has degree $8$ but the associated double cover has some non-simple singularities as well.
So again, we cannot use Theorem~\ref{thm:DegreeCriterion} to conclude that its square root is non-rationalizable.

Analogous computations and results are obtained if one dehomogenizes the polynomials $F_1,F_2,F_3$ with respect to $X_1$ or $X_2$.
For this reason, using the results and techniques of the previous subsections, we cannot prove or disprove the rationalizability of the alphabet $\mathcal{A}$.

Only by using methods that are beyond the scope of this paper, one can see that
the square root $\sqrt{f_1f_2f_3}$ is not rationalizable and hence that the alphabet $\mathcal{A}$ is not rationalizable, where the $f_i$ denote the above mentioned dehomogenizations with respect to $X_3$.

\end{example}

\begin{remark}\label{r:StartingPoint}
Finally, let us stress that, 
when trying to prove non-rationalizability in physics computations,
it is crucial to pick the right starting point for the proof.
To clarify this important subtlety, consider the following alphabet:
\begin{equation*}\label{eq:initsetstartingpoint}
    \mathcal{A}:=\left\{\sqrt{X-1},\sqrt{X-2}\right\}.
\end{equation*}
To rationalize this set, we could proceed as follows: 
\begin{enumerate}
    \item try to rationalize the first square root;
    \item if successful, plug the corresponding substitution into the second square root and try to rationalize the resulting square root;
    \item if successful, compose both substitutions to obtain a single substitution that will rationalize both square roots.
\end{enumerate}
(A more detailed discussion of this procedure can be found in \cite{Besier:2019kco, Besier:2018jen}.)
We start out with the rationalization of the first square root, via the homomorphism $\phi\colon k(X)\to k(X)$ defined by $\phi\colon X\mapsto X^4+1$, hence $\sqrt{\phi(X-1)}=X^2$.
Using $\phi$, the second square root becomes
\begin{equation}\label{eq:nonratstartingpoint}
    \sqrt{X^4-1}\, ,
\end{equation}
giving us a non-rationalizable square root, cf. Theorem~\ref{thm:CriterionForCurves}.

Therefore, one might be tempted to assume that the non-rationalizability of \eqref{eq:nonratstartingpoint} implies non-rationalizability of $\mathcal{A}$.
This assumption is, however, not true:
consider the homomorphism $\psi\colon k(X)\to k(X)$ defined as $\psi\colon X\mapsto X^2+1$.
Then one can easily see that $\psi (X-1)=X^2$ and $\psi(X-2)=X^2-1$.
Notice that $X^2-1$  has degree $2$ and so one can easily rationalize its square root (cf. Corollary~\ref{c:SingProj}).
A suitable substitution is, for example, given by the homomorphism
$\sigma\colon k(X)\to k(X)$ defined by
$\sigma\colon X\mapsto \frac{2X^2}{1-X^2}+1$.
Finally, the composition 
$$
\iota:=\sigma\circ\psi\colon k(X)\to k(X), \; X\mapsto
\frac{2(X^2+1)^2}{1-(X^2+1)^2}+1
$$ 
rationalizes both square roots simultaneously, proving that $\mathcal{A}$ is rationalizable. 

Besides illustrating a way to prove rationalizability for an alphabet, this example gives us the following important insight:
proving non-rationalizability after some substitutions have already been made does, in general, not imply non-rationalizability of the original alphabet.
For Feynman integral computations, this means that one should always prove non-rationalizability as early as possible, i.e., as soon as the square roots arise in the computation.
\end{remark}

\bibliographystyle{elsarticle-harv}
% Include the ".bib" file (generated by bibtex) right here.

\end{document}